\documentclass[a4paper,leqno,11pt,twoside]{amsart}
\usepackage{amsmath,amsfonts,amssymb,amsthm,amscd}
\usepackage[utf8]{inputenc}
\usepackage[english]{babel}
\usepackage[colorlinks=true,citecolor=blue, urlcolor=blue, linkcolor=blue,pagebackref]{hyperref}
\usepackage[top=1.1in,bottom=1.1in,left=1.in,right=1.in]{geometry} 
\usepackage{graphicx}
\usepackage{mathtools}
\usepackage{lipsum}
\input xy
\usepackage{color}
\usepackage[color,matrix,arrow,all]{xy}
\usepackage[all]{xy}

\usepackage{enumerate}
\usepackage{xcolor}
\usepackage{aliascnt}

\theoremstyle{plain}
\newtheorem{thm}{Theorem}[section]

\newaliascnt{propIntr}{thmIntr}

\newaliascnt{corIntr}{thmIntr}

\newaliascnt{lem}{thm}
\newtheorem{lem}[lem]{Lemma}
\aliascntresetthe{lem}

\newaliascnt{cor}{thm}
\newtheorem{cor}[cor]{Corollary}
\aliascntresetthe{cor}

\newaliascnt{prop}{thm}
\newtheorem{prop}[prop]{Proposition}
\aliascntresetthe{prop}

\theoremstyle{definition}

\newaliascnt{rem}{thm}
\newtheorem{rem}[rem]{Remark}
\aliascntresetthe{rem}

\newaliascnt{defn}{thm}

\aliascntresetthe{defn}

\newaliascnt{ex}{thm}

\aliascntresetthe{ex}

\numberwithin{equation}{section}

\def\bP{\ensuremath{\mathbb{P}}}

\def\ZZ{\ensuremath{\mathbb{Z}}}
\def\bC{\ensuremath{\mathbb{C}}}
\def\bF{\ensuremath{\mathbb{F}}}

\def\cA{\ensuremath{\mathcal{A}}}

\def\cC{\ensuremath{\mathcal{C}}}

\def\cL{\ensuremath{\mathcal{L}}}

\def\cO{\ensuremath{\mathcal{O}}}

\DeclareMathOperator{\Pic}{Pic}

\DeclareMathOperator{\Sing}{Sing}

\DeclareMathOperator{\NS}{NS}

\DeclareMathOperator{\im}{Im}

\DeclareMathOperator{\Sp}{Sp}

\DeclareMathOperator{\mult}{mult}

\DeclareMathOperator{\diag}{diag}

\definecolor{applegreen}{rgb}{0.55, 0.71, 0.0}

\title{Hyperelliptic curves  on \((1,d)\) abelian surfaces for $d\leq 15$}
\author[P. Bor\'owka]{Pawe\l{} Bor\'owka}
\address{Pawe\l{} Bor\'owka \newline Institute of Mathematics, Jagiellonian University in Krak\'ow\\ ul. prof. Stanisława Łojasiewicza 6, 
30-348 Kraków, Poland}
\email{pawel.borowka@uj.edu.pl}

\author[P. Grossi]{Paolo Grossi}
\address{Paolo Grossi \newline Dipartimento di Matematica ”F. Casorati”,
Universit\`{a} di Pavia, Via Adolfo Ferrata 5, 27100 Pavia, Italy}
\email{pgrossi.falcone@gmail.com}

\author[A. Shatsila]{Anatoli Shatsila}
\address{Anatoli Shatsila \newline Doctoral School of Exact and Natural Sciences, Jagiellonian University in Krak\'ow \\ ul. prof. Stanisława Łojasiewicza 6, 
30-348 Kraków, Poland}
\email{anatoli.shatsila@doctoral.uj.edu.pl}

\date{}

\begin{document}
\begin{abstract}
    Bryan et al. in \cite{BryanOberdieckPandharipandeYin2015Curve} predicted the number of hyperelliptic curves on general abelian surfaces under a transversality condition. Some of the numbers have already been confirmed. In this paper we construct singular hyperelliptic curves on $(1,d)$ polarised abelian surfaces for $d\leq 15$ confirming a few predicted numbers, namely of (geometric) genus $g=d$ and of geometric genus $k+3$ on $(1,2k)$ and $(1,2k+1)$ polarised surfaces. Moreover, we have found that all genus 7 hyperelliptic curves on general $(1,8)$ polarised surfaces are double coverings of curves locally embedded into $(1,4)$ polarised surfaces. More interestingly, there is a unique genus 9 hyperelliptic curve (up to translation) locally embedded into a general $(1,12)$ polarised surface and it is a Klein covering of a curve locally embedded into a $(1,3)$ polarised surface. 
\end{abstract}
\maketitle
\section{Introduction}
Curves that can be seen as divisors on surfaces have been intensively studied. The earliest example is the famous genus-degree formula for plane curves. Another well-known example is the fact that a general principally polarised abelian surface is the Jacobian of a curve that is the support of the theta divisor. If the theta divisor is not smooth, then it is a union of two elliptic curves and the abelian variety is a principally polarised product, see \cite{BirkenhakeLange1992complex}.

For non-principally polarised abelian surfaces there is a non-trivial linear series and we have results both on general members and special ones. For example, for $d=2$, it is classically known that a smooth genus 3 curve can be embedded into a $(1,2)$ polarised abelian surface if and only if it is a double covering of an elliptic curve. Moreover, there are precisely 3 smooth hyperelliptic curves on a general $(1,2)$ polarised abelian surface.
The Brill-Noether theory of curves on $(1,d)$ polarised abelian surfaces has been investigated in \cite{KLM} and completed in \cite{BayerLi}. 

Curves of genus $g\geq 3$ that certainly can be called special are hyperelliptic curves. Recall that the Kummer surface $Km(A)$ is the quotient of an abelian surface $A$ by the involution $(-1)_A$. Then there is a one-to-one correspondence between rational curves on the Kummer surface and hyperelliptic curves that are invariant under $(-1)$. Using Gromov-Witten invariants, in \cite{BryanOberdieckPandharipandeYin2015Curve} the authors computed (conditionally) the numbers of (geometric) genus $g$ hyperelliptic curves in the linear system of a $(1,d)$ polarisation on a very general abelian surface. Independently, in \cite{BorowkaSankaran2017hyperelliptic, BorowkaOrtega2019hyperelliptic} the authors characterised smooth hyperelliptic curves that can be embedded into abelian surfaces.
Recently, in \cite{GrossiMoretti2025explicit}, a construction of hyperelliptic genus 6 curve on $(1,6)$ polarised abelian surface was crucial to construct a surface of general type with low invariants.

Motivated by the above research, we started an investigation of singular hyperelliptic curves on abelian surfaces. 
Note that, by Andreotti-Mayer theory \cite[\S 6.4]{ACGH}, if you fix a (classical Riemann) theta function on the moduli space, then its zero set on a general abelian variety will be smooth. Even when we define a curve as a zero set of a theta function, the main technical difficulty is to show it has prescribed singularities and geometric genus.


The first result of the paper is an explicit construction of hyperelliptic curves with one node in the linear system of a very general polarised abelian surface, confirming the numbers in \cite[p.10 Table 1]{BryanOberdieckPandharipandeYin2015Curve} and reproving the result from \cite{GrossiMoretti2025explicit}.
\begin{thm}[Thm \ref{thm:GenusgCurves}]
    Let $(A,L)$ be a very general $(1,d)$ polarised abelian surface. Then 
    \begin{enumerate}
        \item For $d=2$ there are precisely 12 hyperelliptic genus 2 curves in $|L|$.  
        \item For $d=3$ there are precisely 90 hyperelliptic genus 3 curves in $|L|$.  
        \item For $d=4$ there are precisely 192 hyperelliptic genus 4 curves in $|L|$.  
        \item For $d=5$ there are precisely 150 hyperelliptic genus 5 curves in $|L|$. 
        \item For $d=6$ there are precisely 36 hyperelliptic genus 6 curves in $|L|$.
        \item For $d\geq 7$ there are no hyperelliptic genus $d$ curves in $|L|$.
    \end{enumerate}
\end{thm}
The idea of the proof is to construct special curves with the node at $0$ and show that they are hyperelliptic. Together with an upper bound, this gives the result.

Apart from the $d=g$ case, we have considered geometric genus $k+3$ hyperelliptic curves on $(1,2k)$ and $(1,2k+1)$ abelian surfaces. We have confirmed some numbers of \cite[p.10 Table 1]{BryanOberdieckPandharipandeYin2015Curve} as the following theorems show. Note that we use the usual convention that $\binom nm=0$ for $m>n$ or $m<0$ and $\binom n0=1$.
\begin{thm}[Thm \ref{thmodd} and Rem \ref{BS}]
    Let $(A, L)$ be a general $(1,d)$ polarised abelian surface for $d=2k+1$. Then there are exactly $d^2\cdot\binom{6}{k-1}$ hyperelliptic curves of geometric genus $k+3$ in the linear series $|L|$. 
\end{thm}
The result reproves $d\leq5$ cases, confirms the numbers for $d=7,9$ and covers $d\geq11$ that are not listed in \cite[p.10, Table 1]{BryanOberdieckPandharipandeYin2015Curve}. The idea of the proof is to show that there is a correspondence between hyperelliptic curves and effective symmetric divisors that pass through ten $2$-torsion points with multiplicity 1 and  have $k-1$ nodes in the other six $2$-torsion points.

For an even $d$ we have a similar formula.
\begin{thm}[Thm \ref{nmbrdeven}]
    Let $(A, L)$ be a very general $(1,d)$ polarised abelian surface for $d=2k$. Then there are exactly $\frac{d^2}4\cdot\binom{4}{k-2}$ hyperelliptic curves of geometric genus $k+3$ in the linear series $|L|$.
\end{thm}
The proof here relies also on the fact that $K(L)\cap A[2]$ consists of 4 points and we want divisors to pass through twelve $2$-torsion points with multiplicity 1 and have nodes in $k-2$ points out of remaining four.

We have been able to describe the geometry of curves for some $d$'s.
Note that for $d$ even, there are curves in a linear system that are invariant under a translation by a $2$-torsion point. Such examples are smooth hyperelliptic curves on $(1,2)$ polarised surfaces that are double coverings of genus 2 curves and similarly hyperelliptic curves on $(1,4)$ polarised surfaces are double coverings of the ones embedded into $(1,2)$ and hence Klein coverings of genus 2 curves. It turns out that for singular curves the analogous result holds as the following theorems show.
\begin{thm}[Thm \ref{thm:1,8}]
	Let \((A,L)\) be a very general \((1,8)\)-polarised abelian surface. Then there are precisely \(96\) geometric genus \(7\) hyperelliptic curves in \(|L|\) and all of them are double covers of geometric genus 4 curves locally embedded into $(1,4)$ polarised surfaces.
\end{thm}
\begin{thm}[Thm \ref{nmbrdeven} and Cor \ref{corKlein}]
    Let $(A,L)$ be a very general $(1,12)$ polarised abelian surface. Then there are precisely 36 geometric genus $9$ hyperelliptic curves in $|L|$. Moreover, each such curve is an \'etale Klein covering of a genus 3 curve on a $(1,3)$ polarised surface.
\end{thm}
Let us explain why the above theorems are interesting. If a hyperelliptic curve $C$ is locally embedded into an abelian surface then it is not hard to construct its Klein covering that will be embedded into another abelian surface. However, the covering need not be hyperelliptic. On the other hand, by \cite{BO20} we can construct hyperelliptic Klein coverings of $C$ using $2$ torsion points on the Jacobian $JC$ of the form $w_1-w_2, w_1-w_3, w_2-w_3$ for some Weierstrass points $w_1,w_2,w_3$ but in general such a covering will not be embedded into a surface. The above theorems show that it is possible to combine both constructions to get hyperelliptic Klein coverings locally embedded into abelian surfaces. 

The importance of our results lies in the fact that we are not aware of any explicit constructions of singular curves on general abelian surfaces. In the paper we not only construct such curves but also show their geometry in $d=8$ and $d=12$ cases. However, it is important to note a very recent preprint \cite{honigs2026approachcurvesabeliansurfaces} that constructs singular curves using completely different techniques, namely Fourier-Mukai transforms. 

The paper is organised as follows. In Section \ref{prel} we recall necessary facts about symmetric line bundles and local embeddings of hyperelliptic curves. In Section \ref{Sgenusg1g} we construct genus $g$ hyperelliptic curves with 1 node on $(1,g)$ polarised abelian surfaces. Section \ref{SectOdd} is devoted to hyperelliptic curves locally embedded into $(1,d)$ polarised surfaces for an odd $d\leq 15$ and Section \ref{secteven} describes the even case together with the description for $d=8$ and $d=12$.

\subsection*{Acknowledgements.} 
P. Bor\'owka has been supported by the Polish National Science Centre project number 2024/54/E/ST1/00330.  A. Shatsila has been supported by the Polish National Science Centre project number 2024/53/N/ST1/01634. Most of the results were obtained during P. Grossi's visits in Krak\'ow in 2025 and 2026. P. Grossi would like to thank Jagiellonian University in Krak\'ow for hospitality.

\section{Preliminaries}\label{prel}

In this section we collect some results about line bundles on abelian surfaces invariant under the $(-1)$-action. All results of this section are consequences of Chapter 4 of \cite{BirkenhakeLange1992complex}. Some results can also be found in \cite[\S 2]{BM}. Since we prefer the coordinate description of line bundles we include proofs. We use the notation from the aforementioned book.

\subsection{Symmetric line bundles and $2$ torsion points}
To fix the notation, for a positive integer $d$ we always consider a $(1,d)$ polarised abelian surface $(A,L)$ with $A = V /\Lambda$ and $L = L(H,\chi)$ for a nondegenerate Hermitian form $H$ on a complex vector space $V$ and a semicharacter $\chi$, and we put $E=\im H$. We also fix the decomposition $\Lambda = \langle e_1, e_2 \rangle \oplus \langle f_1, f_2 \rangle$ for $L$ with $E(e_1, f_1)=1$ and $E(e_2, f_2) = d$, and write $B=(e_1,e_2,f_1,f_2)$. Then $\Lambda(L)=\langle e_1,\frac1de_2,f_1,\frac1df_2\rangle$ and $K(L)=\Lambda(L)/\Lambda$.

Let $L_0$ be the line bundle of characteristic $0$ with respect to the decomposition. Its semicharacter is
\[
\chi_0(\lambda)=\exp\bigl(\pi i\,E(\lambda_1,\lambda_2)\bigr),\qquad \lambda=\lambda_1+\lambda_2,\ \ \lambda_1\in\langle e_1,e_2\rangle,\ \ \lambda_2\in\langle f_1,f_2\rangle .
\]
For $c\in V$ the line bundle $t^*_{\overline c}L_0$ has semicharacter $\chi_0\cdot\exp\bigl(2\pi i\,E(c,\cdot)\bigr)$, and we say that it has \emph{characteristic} $c$; the characteristic is well defined modulo $\Lambda(L)$. The symmetric line bundles in $\Pic^H(A)$ are exactly the sixteen bundles $t^*_{\overline c}L_0$ with
\begin{equation}\label{eq:characteristics}
c=\alpha_1e_1+\alpha_2e_2+\beta_1f_1+\beta_2f_2,\qquad \alpha_1,\beta_1\in\{0,\tfrac12\},\quad\alpha_2,\beta_2\in\{0,\tfrac1{2d}\}.
\end{equation}
If $d$ is odd, then $\frac1{2d}e_2\equiv\frac12e_2$ and $\frac1{2d}f_2\equiv\frac12f_2$ modulo $\Lambda(L)$, and we use the representatives $\frac12e_2$, $\frac12f_2$ instead. In every case $2\alpha_1,2\beta_1,2d\alpha_2,2d\beta_2\in\ZZ$. Here and below we identify the class modulo $\Lambda(L)$ of a characteristic with its representative.

We write the $2$ torsion points of $A$ in coordinates as
\[
x(u,v):=\tfrac12\bigl(u_1e_1+u_2e_2+v_1f_1+v_2f_2\bigr)\bmod\Lambda,\qquad (u,v)=(u_1,u_2,v_1,v_2)\in\bF_2^4 .
\]
By $K(L)[2]$ we denote the subgroup of $2$ torsion points of $K(L)$. If $d$ is even, then $K(L)[2]=\{x(u,v) : u_1=v_1=0\}$; if $d$ is odd, then $K(L)[2]=0$.

Recall that a symmetric effective divisor $D$ is such that $(-1)^*D=D$. For such $D$ on $A$ we denote by $A_2^{+}(D)$ (resp.\ $A_2^{-}(D)$) the set of $2$ torsion points $x \in A$ such that $\mult_x(D)$ is even (resp.\ odd). We call a divisor $D$ even (resp.\ odd) if $0 \in A_2^+(D)$ (resp.\ $0 \in A_2^-(D)$).

For a symmetric line bundle $M$ we denote by $H^0_+(M)$ and $H^0_-(M)$ the eigenspaces for $\pm1$ of the normalised involution $(-1)^*_M$ on $H^0(M)$. We regard sections of line bundles on $A$ as theta functions on $V$. Then $(-1)^*_M$ is $\vartheta\mapsto\vartheta\circ(-1_V)$, so $H^0_+(M)$ (resp.\ $H^0_-(M)$) consists of the even (resp.\ odd) theta functions \cite[Ch.~4]{BirkenhakeLange1992complex}.

\begin{lem}\label{lem:parity}\cite[Lem 4.7.1]{BirkenhakeLange1992complex}
Let $M$ be a symmetric line bundle and $0\neq\sigma\in H^0(M)$.
\begin{enumerate}
\item If $\sigma\in H^0_+(M)$ (resp.\ $\sigma\in H^0_-(M)$), then $Z(\sigma)$ is symmetric and even (resp.\ odd).
\item If $Z(\sigma)$ is symmetric, then $\sigma\in H^0_+(M)\cup H^0_-(M)$.
\end{enumerate}
\end{lem}

\begin{proof}
If $\sigma\circ(-1_V)=\pm\sigma$, then $Z(\sigma)$ is symmetric, and the Taylor expansion of $\sigma$ at $0$ contains only terms of even (resp.\ odd) degree, so $\mult_0Z(\sigma)$ is even (resp.\ odd). Conversely, if $Z(\sigma)$ is symmetric, then $\sigma\circ(-1_V)$ is a theta function for the same factor of automorphy with the same zero divisor. Hence $\sigma\circ(-1_V)=\kappa\sigma$ for some $\kappa\in\bC^*$, and $\kappa^2=1$.
\end{proof}

For a symmetric line bundle $M=L(H,\chi)$ consider the function
\[
q_M\colon A[2]\to\{\pm1\},\qquad q_M(\overline x)=\chi(2x).
\]
By \cite[Proposition 4.7.2]{BirkenhakeLange1992complex}, $q_M(x)=(-1)^{\mult_xD-\mult_0D}$ for every symmetric divisor $D\in|M|$. In particular $q_M$ does not depend on the decomposition, and for a symmetric $D\in|M|$
\begin{equation}\label{eq:Aminus}
A_2^-(D)=q_M^{-1}(-1)\ \text{ if $D$ is even},\qquad A_2^-(D)=q_M^{-1}(1)\ \text{ if $D$ is odd}.
\end{equation}
For a symplectic basis $B'=(e_1',e_2',f_1',f_2')$ of $\Lambda$ of type $(1,d)$ we write $q_M$ in coordinates as the function $\bar q_{M,B'}\colon\bF_2^4\to\bF_2$ given by
\[
(-1)^{\bar q_{M,B'}(u,v)}=q_M\bigl(\tfrac12(u_1e_1'+u_2e_2'+v_1f_1'+v_2f_2')\bigr).
\]

\begin{lem}\label{lem:qformula}
Let $c$ be as in \eqref{eq:characteristics}. Then
\[
\bar q_{t^*_{\overline c}L_0,B}=\bar q_c:=u_1v_1+d\,u_2v_2+a_1u_1+b_1v_1+a_2u_2+b_2v_2,
\]
where $(a_1,b_1,a_2,b_2)\equiv(2\beta_1,2\alpha_1,2d\beta_2,2d\alpha_2)\pmod 2$. In particular, $M\mapsto\bar q_{M,B}$ is a bijection between the sixteen symmetric line bundles in $\Pic^H(A)$ and the functions $u_1v_1+d\,u_2v_2+\ell$, where $\ell$ is a linear form on $\bF_2^4$.
\end{lem}

\begin{proof}
The semicharacter of $t^*_{\overline c}L_0$ gives $q_{t^*_{\overline c}L_0}(x)=\chi_0(2x)\exp\bigl(4\pi i\,E(c,x)\bigr)$. For $x=x(u,v)$ we get $\chi_0(2x)=\exp\bigl(\pi i(u_1v_1+du_2v_2)\bigr)$ and $4E(c,x)=2\alpha_1v_1+2d\alpha_2v_2-2\beta_1u_1-2d\beta_2u_2$, which gives the formula. The linear part determines $2\alpha_1,2\beta_1,2d\alpha_2,2d\beta_2$ modulo $2$, hence $c$ modulo $\Lambda(L)$, and every linear form occurs.
\end{proof}

Let $d$ be odd. Then $\bar q_c$ is the sum of the quadratic forms $u_1v_1+a_1u_1+b_1v_1$ and $u_2v_2+a_2u_2+b_2v_2$ on the two coordinate planes, so it is a non-degenerate quadratic form on $\bF_2^4$ with Arf invariant $a_1b_1+a_2b_2$. We call a symmetric line bundle $M\in\Pic^H(A)$ \emph{odd} if this Arf invariant is $1$, and \emph{even} otherwise. By the proof of Proposition~\ref{prop:twotorsionofdivisors} below, $M$ is odd if and only if $\#q_M^{-1}(-1)=10$, so this notion does not depend on the decomposition. By Lemma~\ref{lem:qformula}, the odd line bundles are the bundles $t^*_{\overline c}L_0$ with $c$ in
\[
S := \{\tfrac{1}{2}e_1 + \tfrac12f_1,\ \tfrac{1}{2}e_1 + \tfrac12f_1+\tfrac{1}{2}e_2,\ \tfrac{1}{2}e_1 + \tfrac12f_1+\tfrac{1}{2}f_2,\ \tfrac{1}{2}e_2 + \tfrac{1}{2}f_2,\ \tfrac{1}{2}e_2 + \tfrac{1}{2}f_2+\tfrac12e_1,\ \tfrac{1}{2}e_2 + \tfrac{1}{2}f_2+\tfrac12f_1\}.
\]
In particular $L_0$ is even. Now we compute the cardinality of $A^\pm_2(D)$.
\begin{prop}
\label{prop:twotorsionofdivisors}
Let $(A,L)$ be a $(1,d)$ polarised abelian surface and let $D$ be a symmetric even divisor on $A$ with $\mathcal{O}_A(D) = t^*_{\overline{c}}L_0$.
\begin{enumerate}
\item Assume $d$ is even. If $c \in \{0, \frac12 e_1, \frac12 f_1\}$, then $\# A_2^-(D) = 4$. If $c = \frac12 e_1 + \frac12 f_1$, then $\# A_2^-(D) = 12$ and $A_2^+(D)=K(L)[2]$. For the other $12$ characteristics, $\# A_2^-(D) = 8$ and $\# (A_2^-(D) \cap K(L)[2]) = 2$.
\item Assume $d$ is odd. Then $\# A_2^-(D) = 10$ if $c \in S$, and $\# A_2^-(D) = 6$ if $c \notin S$.
\end{enumerate}
If $D$ is a symmetric odd divisor with $\cO_A(D)=t^*_{\overline c}L_0$, then $A_2^-(D)$ is the complement in $A[2]$ of the corresponding set for an even divisor. Hence $\#A_2^-(D)$ equals $12$, $4$ and $8$ in the three cases of (1), and $6$ and $10$ in the two cases of (2); moreover $\#(A_2^-(D)\cap K(L)[2])=2$ for the twelve characteristics in the last case of (1).
\end{prop}

\begin{proof}
By \eqref{eq:Aminus} and Lemma~\ref{lem:qformula}, $A_2^-(D)=\{x(u,v) : \bar q_c(u,v)=1\}$. Write $\bar q_c=Q_1+Q_2$ with $Q_1=u_1v_1+a_1u_1+b_1v_1$ and $Q_2=d\,u_2v_2+a_2u_2+b_2v_2$, and let $N_i$ be the number of vectors in $\bF_2^2$ at which $Q_i$ takes the value $1$. A direct check gives $N_1=1+2a_1b_1$. Similarly $N_2=1+2a_2b_2$ if $d$ is odd, while for $d$ even $N_2=0$ if $(a_2,b_2)=(0,0)$ and $N_2=2$ otherwise. Hence
\[
\#A_2^-(D)=N_1(4-N_2)+(4-N_1)N_2 .
\]
Let $d$ be even. If $(a_2,b_2)=(0,0)$, that is $\alpha_2=\beta_2=0$, this equals $4N_1$, which is $4$ for $c\in\{0,\frac12e_1,\frac12f_1\}$ and $12$ for $c=\frac12e_1+\frac12f_1$. In the latter case $\bar q_c(u,v)=0$ exactly when $u_1=v_1=0$, so $A_2^+(D)=K(L)[2]$. If $(a_2,b_2)\neq(0,0)$, we get $8$, and the restriction of $\bar q_c$ to $K(L)[2]=\{u_1=v_1=0\}$ is the non-zero linear form $a_2u_2+b_2v_2$, which takes the value $1$ at exactly two points.

Let $d$ be odd. Then $N_1,N_2\in\{1,3\}$, and the formula gives $10$ if exactly one of $N_1,N_2$ equals $3$, that is if $a_1b_1+a_2b_2=1$, i.e.\ $c\in S$, and $6$ otherwise.

Finally, if $D$ is odd, then $A_2^-(D)=q_{\cO_A(D)}^{-1}(1)$ by \eqref{eq:Aminus}, which is the complement of the set for an even divisor.
\end{proof}

\begin{lem}
\label{lem:dimensionofh0}
    Let $(A, L)$ be a $(1,d)$ polarised abelian surface. Then
    \begin{enumerate}
        \item If $d=2$ then $h^0_+(t_{\overline{c}}^*L_0) = 2$ for $c \in \{0, \frac12 e_1, \frac12 f_1\}$.
        \item If $d=3$ then $h^0_+(t_{\overline{c}}^*L_0) = 2$ for $c \in \frac12\Lambda(L) \setminus S$.
        \item If $d = 4$ then $h^0_+(t_{\overline{c}}^*L_0) = 2$ for $c \in \frac12\Lambda(L) \setminus \{0, \frac12 e_1, \frac12 f_1, \frac12e_1+\frac12f_1\}$.
        \item If $d=5$ then $h^0_+(t_{\overline{c}}^*L_0) = 2$ for $c \in S$.
        \item If $d=6$ then $h^0_+(t_{\frac12e_1 + \frac12f_1}^*L_0) = 2$.
        \item For $d=2k,\ k\geq 4$, we have $h^0_+(t_{\frac12e_1 + \frac12f_1}^*L_0) = k-1$.
        \item For $d=2k+1,\ k\geq 3$, we have     $h^0_-(L_0) = k$.
    \end{enumerate}
\end{lem}
\begin{proof}
    This follows directly from \cite[Exercise 4.12.11]{BirkenhakeLange1992complex}.
\end{proof}

For $2\leq d\leq 6$ let $\cC_d$ be the set of characteristics $c$ such that $\#A_2^-(D)=2d$ for a symmetric even divisor $D$ with $\cO_A(D)=t^*_{\overline c}L_0$. By Proposition~\ref{prop:twotorsionofdivisors},
\[
\cC_2=\{0,\tfrac12e_1,\tfrac12f_1\},\qquad
\cC_3=\tfrac12\Lambda(L)\setminus S,\qquad
\cC_4=\tfrac12\Lambda(L)\setminus\{0,\tfrac12e_1,\tfrac12f_1,\tfrac12e_1+\tfrac12f_1\},
\]
\[
\cC_5=S,\qquad \cC_6=\{\tfrac12e_1+\tfrac12f_1\},
\]
so $\#\cC_d=3,10,12,6,1$ for $d=2,\dots,6$. By Lemma~\ref{lem:dimensionofh0}, $h^0_+(t^*_{\overline c}L_0)=2$ for every $c\in\cC_d$.

\subsection{Changing the decomposition}
A change of decomposition changes the characteristic of a symmetric line bundle. Since $q_M$ is intrinsic, this is easy to describe in terms of the functions $\bar q_{M,B}$. The main aim of the following technical lemma is to show that $\{\bar q_c, c\in \cC_d\}$ form a single orbit in each case $d=2,\dots6$.

\begin{lem}\label{lem:orbits}
Let $(A,L)$ be a $(1,d)$ polarised abelian surface.
\begin{enumerate}[(a)]
\item For $\gamma\in\Sp(\Lambda,E)$ and a symmetric $M\in\Pic^H(A)$ we have $\bar q_{M,\gamma B}=\bar q_{M,B}\circ\bar\gamma$, where $\bar\gamma\in\mathrm{GL}_4(\bF_2)$ is the reduction modulo $2$ of the matrix of $\gamma$ in the basis $B$.
\item Let $2\leq d\leq 6$ and let $\Gamma\subset\Sp(\Lambda,E)$ be the subgroup generated by
  \[
  T_i\colon f_i\mapsto f_i+e_i,\qquad
  J_i\colon (e_i,f_i)\mapsto(f_i,-e_i)\qquad(i=1,2),
  \]
  \[
  M_1\colon e_1\mapsto e_1+e_2,\ f_2\mapsto f_2-df_1,\qquad
  M_2\colon f_1\mapsto f_1+f_2,\ e_2\mapsto e_2-de_1,
  \]
  where basis vectors that are not mentioned are fixed. Then
  $\{\bar q_c : c\in\cC_d\}$ is a single orbit of the action
  $q\mapsto q\circ\bar\gamma$ of $\Gamma$.
\end{enumerate}
\end{lem}

\begin{proof}
(a) The point with coordinates $(u,v)$ with respect to $\gamma B$ is the image under $\gamma$ of the point with the same coordinates with respect to $B$, so its $B$-coordinates are $\bar\gamma(u,v)$. Since $q_M$ is intrinsic, the claim follows.

(b) One checks directly that the six generators preserve $E$. Write the functions of Lemma~\ref{lem:qformula} as $q=u_1v_1+d\,u_2v_2+a_1u_1+b_1v_1+a_2u_2+b_2v_2$ and record $q$ by $(a_1,b_1;a_2,b_2)\in\bF_2^4$. By the proof of Proposition~\ref{prop:twotorsionofdivisors}, $\{\bar q_c : c\in\cC_d\}$ corresponds to the following sets of $(a_1,b_1;a_2,b_2)$:
\[
\begin{array}{lll}
d=2\colon & (a_2,b_2)=(0,0),\ a_1b_1=0; \qquad &
d=3\colon\ a_1b_1=a_2b_2;\\[2pt]
d=4\colon & (a_2,b_2)\neq(0,0); &
d=5\colon\ a_1b_1\neq a_2b_2;\\[2pt]
d=6\colon & (a_2,b_2)=(0,0),\ a_1b_1=1. &
\end{array}
\]
A computation with the reductions modulo $2$ gives the following action on $(a_1,b_1;a_2,b_2)$:
\begin{itemize}
\item $\bar T_1$ and $\bar J_1$ send $(a_1,b_1)$ to $(a_1,a_1+b_1+1)$ and $(b_1,a_1)$, and fix $(a_2,b_2)$. They permute the three pairs with $a_1b_1=0$ transitively and fix $(1,1)$.
\item $\bar T_2$ and $\bar J_2$ fix $(a_1,b_1)$. For $d$ odd they act on $(a_2,b_2)$ as $\bar T_1,\bar J_1$ act on $(a_1,b_1)$. For $d$ even they send $(a_2,b_2)$ to $(a_2,a_2+b_2)$ and $(b_2,a_2)$, and so permute the three non-zero pairs transitively.
\item $\bar M_1\colon (a_1,b_1;a_2,b_2)\mapsto(a_1+a_2,\,b_1;\,a_2,\,b_2+db_1)$ and $\bar M_2\colon (a_1,b_1;a_2,b_2)\mapsto(a_1,\,b_1+b_2;\,a_2+da_1,\,b_2)$.
\end{itemize}
For $d=2$ and $d=6$ the claim follows from the first item. For $d=4$, first use $\bar T_2,\bar J_2$ to reach any prescribed non-zero $(a_2,b_2)$ without changing $(a_1,b_1)$. If $a_2=1$ (resp.\ $b_2=1$), then $\bar M_1$ (resp.\ $\bar M_2$) changes $a_1$ (resp.\ $b_1$) and fixes $(a_2,b_2)$; together with the first item this reaches all four pairs $(a_1,b_1)$. For $d$ odd, the groups generated by $\bar T_i,\bar J_i$ have the orbits determined by $(a_1b_1,a_2b_2)$, and $\bar M_1$ sends $(1,1;1,1)$ to $(0,1;1,0)$ and $(0,1;1,1)$ to $(1,1;1,0)$. This joins the two orbits with $a_1b_1=a_2b_2$ (for $d=3$) and the two orbits with $a_1b_1\neq a_2b_2$ (for $d=5$).
\end{proof}

\subsection{Local embeddings of hyperelliptic curves}
Let us describe a well-known construction. By a curve locally embedded into an abelian surface we mean a birational (onto the image) map $f:C\to A$ from a smooth curve $C$ to an abelian surface $A$. We call $f$ \textit{a symmetric local embedding} if $f(C)$ is invariant under $(-1)_A$.

If $C$ is hyperelliptic then one can choose a Weierstrass point $w\in C$ and consider an Abel-Jacobi embedding $\alpha_w:C\to JC$. Then, by Universal Property of Jacobians there is a homomorphism $F:JC\to A$ such that the following diagram commutes

\begin{equation}\label{UniJac}
\xymatrix@R=.9cm@C=1cm{
 C\ar[d]_{\alpha_w} \ar[r]^{ f} &  A  \ar[d]^{t_{-f(w)}}  \\
JC \ar[r]_{F} &  A 
}
\end{equation}

Now, note that $\alpha_w(C)$ is invariant under $(-1)_{JC}$ (because $-1$ restricts to the hyperelliptic involution) hence $F(\alpha_w(C))$ is invariant under $(-1)_A$. This shows that firstly, up to translation, we can always assume that the image $f(C)$ is symmetric and secondly, if $f(w)=0$ for a Weierstrass point, then $f(C)$ is symmetric. Since the translation by a 2-torsion point of a symmetric divisor is still symmetric, we only need to check if $f(w)\in A[2]$ for some Weierstrass point $w$. In particular, in the paper we will always assume $f$ is a symmetric local embedding. 
\begin{lem}\label{lemsymodd}
    Let $A$ be a $(1,d)$ polarised abelian surface with $d$ odd. Then $f:C\to A$ is a symmetric local embedding if and only if $f(w)\in A[2]$ for some (hence all) Weierstrass points $w\in C$.
\end{lem}
\begin{proof}
    One direction has already been proved, so assume $f(C)$ is symmetric and $\cO(f(C))$ is of characteristic $c$. Then $c\in A[2d]$ since the characteristic is a 2-torsion point up to the elements of the kernel of the polarising isogeny.    
    Let $w\in C$ be a Weierstrass point and consider $D=t^*_{-f(w)}f(C)$. Firstly, $\cO(D)$ is of characteristic $c-f(w)$ but since it is a symmetric line bundle, its characteristic also satisfies $c-f(w)\in A[2d]$. In particular $f(w)\in A[2d]$. Now, note that $f(C)$ is invariant under $t_{2f(w)}$. To show it, assume $x\in f(C)$. From the symmetry of $f(C)$, we get that $-x\in f(C)$, so $-x-f(w)\in D$. Then, using symmetry of $D$ one gets $x+f(w)\in D$ and hence $x+2f(w)\in f(C)$. Now, $t^*_{2f(w)}|_{f(C)}$ is an automorphism that is either identity or fixed point free that pulls back to a group of fixed point free automorphisms on $C$ so it has to be either identity or an involution, see \cite[p.239]{Ramanan}. Hence $f(w)\in A[4]\cap A[2d]$ for $d$ odd, so $f(w)\in A[2]$ and we are done.   
\end{proof}

\begin{rem}
    Note that $f(w)\in A[2]$ for a Weierstrass point $w$ is not a necessary condition for $f$ to be a symmetric local embedding for an even $d$. Let $(A, L_A)$ be a $(1,4)$ polarised abelian surface and $C_A$ a hyperelliptic Klein covering of a genus 2 curve, see \cite{BorowkaOrtega2019hyperelliptic}. Assume $f:C_A\to A$ satisfies $f(w)=0$. Note that $C_A$ is invariant under the Klein group $K(L_A)[2]$. Let $x\in K(L_A)$ be a 4-torsion point. Then $t_x(C_A)$ is symmetric because $-t_x(c)=-c-x=t_x(-c-2x)$ for any point $c\in C_A$ even though Weierstrass points are mapped to $4$-torsion points.
\end{rem}

Let us show a lemma that will be used throughout the paper as it connects the type of the polarisation, the number of 2-torsion points with odd multiplicity and the geometric genus of a hyperelliptic curve.
\begin{lem}
\label{lem:sing_count}
Let $(A, L)$ be a $(1,d)$ polarised abelian surface and let $D$ be a reduced effective 
symmetric divisor on $A$ whose normalisation $\nu \colon C \to D$ is a
smooth hyperelliptic curve of genus $g$, such that $(-1)_A$ lifts through
$\nu$ to the hyperelliptic involution on $C$. Then
\[ 2g + 2 \;\leq\; \# A_2^-(D) \,+\, 2\bigl(p_a(D) - g\bigr), \]
with equality if and only if every singularity of $D$ is an ordinary node
located at a $2$ torsion point of $A$, with both preimages under $\nu$
being Weierstrass points of $C$.
\end{lem}

\begin{proof}
The hyperelliptic involution $\iota$ maps Weierstrass points to $A[2]$. If it maps to a smooth point
$s$ of $D$, then $\mult_s(D) = 1$, so $s \in A_2^-(D)$, and $\nu$ being
bijective over the smooth locus, distinct Weierstrass points give distinct
points of $A_2^-(D)$. If it maps to a singular point $s$, then $s \in
A[2]$ and the Weierstrass point corresponds to a branch of $D$ at $s$ preserved
by $(-1)$; hence $s$ supports at most $r_s$ such fixed points, where
$r_s$ is the number of branches. Since $r_s \leq \delta_s + 1$ for any
curve singularity, with equality precisely for nodes, and since
$\sum_s \delta_s = p_a(D) - g$, the singular locus supports at most
$2(p_a(D)-g)$ fixed points, with equality forcing every singularity to be
a node at a point of $A[2]$ with both branches preserved. Summing the two
contributions gives the claim.
\end{proof}

\section{Hyperelliptic genus \texorpdfstring{$g$}{g} curves on \texorpdfstring{$(1,g)$}{(1,g)} polarised surfaces}\label{Sgenusg1g}
This section is devoted to hyperelliptic curves with one node. Let us start with the following lemma.
\begin{lem}
\label{lem:symmdivisor}
    Let $(A,L)$ be a $(1,g)$ polarised abelian surface and $D \in |L|$ an irreducible curve on $A$ passing through $g+1$ points of $A[2]$ with multiplicity 1. Then $(-1)^*D = D$. 
\end{lem}

\begin{proof}
    Assume that $(-1)^*(D) \neq D$. Since $(-1)^*(D)$ is algebraically equivalent to $D$, we have $$D \cdot (-1)^*D = D^2=2g.$$ Let $p \in A[2] \cap D$ be a point such that $\operatorname{mult}_p(D) = 1$. Then $p \in (-1)^*D$ and the local intersection multiplicity at $p$ is $$(D \cdot (-1)^*D)_p \geq 2$$ because both $D$ and $(-1)^*D$ pass through the smooth point $p$ with the same tangent line. But then $$2g = D \cdot (-1)^*D \geq \sum_{p \in {A[2] \cap D}} (D \cdot (-1)^*D)_p \geq 2g+2,$$ which is a contradiction.
\end{proof}

\begin{lem}
	\label{lem:odd_mult_points}
	Let $2 \leq g \leq 7$ be a positive integer, $(A, L)$ a $(1,g)$ polarised abelian surface, $C$ a smooth genus-$g$ curve, and $f:C \to A$ a local embedding with \(f(C) \in |L|\). Then $C$ is hyperelliptic if and only if there exists a point $x \in A$ such that $t_x(f(C))$ passes through $2g$ 2-torsion points with multiplicity 1. 
\end{lem}

\begin{proof} 
	Assume $C$ is hyperelliptic and $w$ is a Weierstrass point on $C$. Consider $D:=t_{-f(w)}f(C)$, so the point $w$ is mapped to $0$. Then $D$ is a symmetric divisor on $A$, and $(-1)_A$ lifts through the normalisation $C \to D$ to the hyperelliptic involution on $C$. Since $p_a(D) = g+1,$ $D$ has either one node or one cusp and since $D$ is symmetric the singularity has to be in a 2-torsion point of $A$. If $D$ has a cusp then $\# A_2^-(D) = 2g+1$, which is impossible, since $\# A_2^-(D)$ is always even by Proposition \ref{prop:twotorsionofdivisors}. 
    Thus, $D$ has a node, so there are $2g$ Weierstrass points on $C$ that map to distinct $2$-torsion points of $A$ with multiplicity 1 under $t_{-f(w)}\circ f$, so the claim follows.

	Now, assume that $D := t_x(f(C))$ passes through $2g$ points in $A[2]$ with multiplicity 1.    
    Then $D$ is a symmetric divisor by Lemma \ref{lem:symmdivisor}. Let $\tilde{A} \to A$ be the resolution of $D\subset A$ such that the preimage of $D$ is the curve $C$. Consider the image $C^-$ of $C$ under the composition $C \to \tilde{A} \to A \to A/(-1)$. By the assumptions, the ramification degree of the covering $C \to C^-$ is at least $2g$, hence the Riemann-Hurwitz formula yields $$2g-2 \geq 2(2g(C^-)-2)+2g.$$ This implies that $g(C^-) \leq \frac12$, hence $C^- \simeq \mathbb{P}^1$ and $C$ is a smooth hyperelliptic curve. 
\end{proof}

\begin{lem}
	\label{lem:unique_curve}
    Let $2 \leq g \leq 6$, let $C$ be a smooth hyperelliptic genus-$g$ curve, and let $f: C \to A$ be a local embedding into an abelian surface $A$ such that $\mathcal{O}_A(f(C))$ is of type $(1,g)$. Then 
    \begin{itemize}
        \item[i)] There exists a unique point $z \in A$ such that the divisor $f_0(C)$, where $f_0 := t_z\circ f$, is symmetric with a node at $0$. Moreover, the two preimages of the node are Weierstrass points of $C$. 
        \item[ii)] If $f:C \to A$ and $f':C' \to A$ are local embeddings as in $(i)$ and $\mathcal{O}_A(f(C)) = \mathcal{O}_A(f'(C'))$, then $f(C) = f'(C')$.
    \end{itemize}

\end{lem}
\begin{proof}
Take $x\in A$ such that $t_x\circ f$ is a symmetric local embedding. Then, the assumptions force the equality in Lemma \ref{lem:sing_count}, so we have a unique node in a $2$-torsion point $y$. Then $t_{-y}\circ t_x\circ f$ is a symmetric local embedding with a unique node at $0$, hence $z=x-y$ is unique.

    To show $ii)$, assume that $f(C) \neq f'(C')$. Then $$f(C)\cdot f'(C') = \mathcal{O}_A(f(C))^2 = 2g,$$ however, since $f(C)$ and $f'(C')$ both have a node at 0 and pass through the same $2g$ two-torsion points with multiplicity 1 (since $\mathcal{O}_A(f(C)) = \mathcal{O}_A(f'(C'))$) we have $$f(C) \cdot f'(C') \geq 2g + 4 > 2g,$$ which is a contradiction. Thus, $f(C) = f'(C')$. 
\end{proof}

\begin{cor}
    \label{cor:nocurvesofhighgenus}
    Let $g \geq 7$ and let $(A,L)$ be a $(1,g)$ polarised abelian surface. Then there are no hyperelliptic curves of geometric genus $g$ in the algebraic equivalence class of $L$. 
\end{cor}

\begin{proof}

Suppose that such a hyperelliptic curve $f: C \to A$ exists for some $g \geq 7$. Arguing as in the proof of Lemma \ref{lem:odd_mult_points} we can choose a local embedding $f$ that maps Weierstrass points of C to $A[2]$. Since $p_a(f(C)) = g+1$, the curve $f(C)$ has either a node or a cusp, so it passes through at least $2g$ points of $A[2]$ with multiplicity $1$, so it is symmetric by Lemma \ref{lem:symmdivisor}. Then $\#A_2^-(f(C)) \geq 2g \geq 14$, which is impossible by Proposition \ref{prop:twotorsionofdivisors}.
   
\end{proof}
\begin{rem}
    Corollary \ref{cor:nocurvesofhighgenus} can also be deduced from the emptiness of the Brill-Noether locus proved in \cite{KLM}.
\end{rem}

In the rest of the section we consider the locus of abelian surfaces $(A,H) \in \mathcal{A}_2^{(1,d)}$ satisfying the following two conditions:
\begin{enumerate}
\item[(G1)] $\operatorname{NS}(A) = \mathbb{Z}H$; in particular, every curve in the algebraic equivalence class of $H$ is reduced and irreducible;
\item[(G2)] for every $c\in\cC_d$, the section of $H^0_+(L_c)$ vanishing at $0$ has non-degenerate quadratic term at $0$.
\end{enumerate}

We will show that (G2) is non-empty, hence this locus is a complement of a countable union of proper closed subvarieties, so we say that $(A,H)$ is \emph{very general}, if it lies in this locus. The next lemma is crucial for constructing enough hyperelliptic curves with one node.
 
\begin{lem}\label{lem:existence}
Let $2\leq d\leq 6$, let $(A,H)$ be a very general $(1,d)$ polarised abelian
surface, and let $L=t^*_{\overline c}L_0\in\Pic^H_s(A)$ with $c\in\cC_d$, so
that $h^0_+(L)=2$. Let $\theta_1,\theta_2$ be a basis of $H^0_+(L)$ and put
\[
s:=\theta_2(0)\,\theta_1-\theta_1(0)\,\theta_2 .
\]
Then $s\neq0$, every section in $H^0_+(L)$ vanishing at $0$ is a multiple of
$s$, and $D:=Z(s)$ is an irreducible curve whose only singularity is an
ordinary node at $0$ and whose normalisation $C$ is a hyperelliptic curve of
genus $d$. Moreover, $(-1)_A$ lifts to the hyperelliptic involution of $C$.
In particular, up to a non-zero constant, $s$ is the unique section of
$H^0_+(L)$ whose zero locus is a hyperelliptic curve of geometric genus $d$
with a node at $0$.
\end{lem}
 
\begin{proof}
\emph{Step 1: Proof that $D:=Z(s)$ is an irreducible curve with an ordinary node at $0$}.\\
Let $(A,H)$ satisfy
(G1), so that every curve in $|L|$ is reduced and irreducible. By
Lemma~\ref{lem:parity}, for every $0\neq\sigma\in H^0_+(L)$ the divisor
$Z(\sigma)$ is symmetric and even, so by Formula \eqref{eq:Aminus} the set
$A_2^-(Z(\sigma))=q_L^{-1}(-1)$ has $2d$ elements, because $c\in\cC_d$, and
it does not contain $0$, see also Proposition \ref{prop:twotorsionofdivisors}.
 
Suppose that every element of $H^0_+(L)$ vanishes at $0$. Take a basis $\theta_1,\theta_2\in H^0_+(L)$ and put $D_i:=Z(\theta_i)$. These
are distinct irreducible curves, so they have no common component. Both have
even, hence at least $2$, multiplicity at $0$, and both pass through the $2d$
points of $q_L^{-1}(-1)$. Therefore
\[
2d=L^2=D_1\cdot D_2\;\geq\;2d+\mult_0D_1\cdot\mult_0D_2\;\geq\;2d+4,
\]
a contradiction. Thus evaluation at $0$ is a non-zero linear form on the
$2$-dimensional space $H^0_+(L)$. Hence $(\theta_1(0),\theta_2(0))\neq(0,0)$, by independence
$s\neq0$, and the sections vanishing at $0$ are spanned by $s$. In
particular, the section in (G2)
is $s$ up to a constant. Since $s(0)=0$ and $ds(0)=0$, the Taylor
expansion of $s$ at $0$ starts with this quadratic term. Hence $D$ has an ordinary node at
$0$. The divisor $D=Z(s)$ is irreducible, symmetric and
even, satisfies $\mult_0D=2$, and passes through the $2d$ points of
$A_2^-(D)$.
 
\emph{Step 2: Proof that $C$ is hyperelliptic.} \\ 
Let $\nu\colon C\to D$ be the normalisation. The curve $C$ is smooth,
projective and connected because $D$ is irreducible. Since $D$ is symmetric,
$(-1)_A$ lifts to an involution $j$ of $C$. This involution is non-trivial
because $(-1)_A$ has only finitely many fixed points. Points of $C$ over
$x\in D$ correspond to branches of $D$ at $x$. A point over $x$ is fixed by
$j$ if and only if $x\in A[2]$ and the corresponding branch is
$(-1)_A$-invariant.
\begin{itemize}
\item The two branches of the node at $0$ have distinct tangent lines. Since
  $(-1)_A$ acts on $T_0A$ as $-\mathrm{id}$, it preserves every line through
  the origin, hence each of the two branches. This gives two fixed points of
  $j$ over $0$.
\item Let $x\in A_2^-(D)$. Branches at $x$ that are not invariant are
  exchanged in pairs of equal multiplicity, while the multiplicities of all
  branches at $x$ add up to the odd number $\mult_xD$. Hence some branch at
  $x$ is invariant, which gives at least one fixed point of $j$ over $x$.
\end{itemize}
Since $\#A_2^-(D)=2d$ and $0\notin A_2^-(D)$, the involution $j$ has at least
$2d+2$ fixed points.
 
On the other hand, $p_a(D)=\frac12L^2+1=d+1$ and
\[
g(C)=p_a(D)-\sum_{x\in\operatorname{Sing}D}\delta_x,
\]
where $\delta_x=\dim_{\bC}(\nu_*\cO_C/\cO_D)_x$ is the $\delta$-invariant. It
is at least $1$ at every singular point and equal to $1$ at a node. Hence
$g(C)\leq d$, with equality if and only if the node at $0$ is the only
singularity of $D$. The Riemann--Hurwitz formula for $C\to C/\langle j\rangle$
gives
\[
2d-2\;\geq\;2g(C)-2\;=\;2\bigl(2g(C/\langle j\rangle)-2\bigr)+\#\operatorname{Fix}(j)\;\geq\;-4+2d+2\;=\;2d-2 .
\]
Hence all inequalities are equalities. We get $g(C)=d$,
$C/\langle j\rangle\simeq\bP^1$ and $\#\operatorname{Fix}(j)=2d+2$, and the
node at $0$ is the only singularity of $D$. As $d\geq2$, the curve $C$ is
hyperelliptic and $j$ is its hyperelliptic involution. Its Weierstrass points
are the two points over $0$ and the points over $A_2^-(D)$.
 
\emph{Step 3: The property (G2) holds at some point.}\\ It
suffices to find, for each $c\in\cC_d$, one surface satisfying (G1) on which
the section of $H^0_+(L_c)$ vanishing at $0$ has an ordinary node at $0$, that
is, has non-degenerate quadratic term at $0$.
 
For $\tau\in\mathfrak h_2$ let $\Lambda_\tau=\tau\ZZ^2+\diag(1,d)\ZZ^2$ and
$A_\tau=\bC^2/\Lambda_\tau$. Let $B_\tau=(e_1,e_2,f_1,f_2)$ be the
decomposition of $\Lambda_\tau$ given by the columns of the period matrix
$(\tau,\diag(1,d))$, that is $e_i=\tau\varepsilon_i$, $f_1=\varepsilon_1$,
$f_2=d\varepsilon_2$, where $\varepsilon_1,\varepsilon_2$ is the standard basis
of $\bC^2$. For $c\in\cC_d$ let $L_{c,\tau}:=t^*_{\overline c}L_0$ on
$A_\tau$, with characteristics taken with respect to $B_\tau$.
 
Apply \cite[Theorem 1.6(i)]{KLM} with $\delta=1$ and $k=2$. Then the
geometric genus is $d$ and $\alpha=\lfloor(d-1)/4\rfloor$, and the numerical
condition there reads $1\geq\alpha(d-3-2\alpha)$, which holds for
$2\leq d\leq6$. Hence a general $(1,d)$ polarised abelian surface carries 
an irreducible hyperelliptic curve with exactly one
node.
Choose a surface $A_{\tau_A}$ with $\NS(A_{\tau_A})=\ZZ H$, so that it
satisfies (G1), and let $C_A$ be such a curve.
 
By Lemma~\ref{lem:unique_curve}(i) and its proof, a translate $D_A$ of $C_A$ is
symmetric with its node at $0$, and $(-1)_A$ lifts through the normalisation
to the hyperelliptic involution. The node is the only singular point of $D_A$,
since $p_a(D_A)-d=1$. Hence every $x\in A_2^-(D_A)$ is a smooth point of $D_A$
fixed by $(-1)_A$, so its preimage is a Weierstrass point. Conversely, the $2d$
Weierstrass points not lying over $0$ map to distinct smooth points of $D_A$ in
$A[2]$, which lie in $A_2^-(D_A)$. Thus $\#A_2^-(D_A)=2d$. Since $D_A$ is even,
Proposition~\ref{prop:twotorsionofdivisors} gives $\cO_A(D_A)=L_{c_0,\tau_A}$
for some $c_0\in\cC_d$. By Lemma~\ref{lem:parity}, the theta function defining
$D_A$ lies in $H^0_+(L_{c_0,\tau_A})$, and it vanishes at $0$. So by Step~1 it
is a multiple of the section $s$ of $L_{c_0,\tau_A}$. As $D_A$ has an ordinary
node at $0$, (G2) holds for $c_0$ on $A_{\tau_A}$.
 
Now let $c\in\cC_d$ be arbitrary. By Lemma~\ref{lem:orbits}(b) there is an
element $\gamma$ of the group generated by $T_1,J_1,T_2,J_2,M_1,M_2$ with
$\bar q_{c_0}\circ\bar\gamma=\bar q_c$. By Lemmas~\ref{lem:qformula}
and~\ref{lem:orbits}(a), $\bar q_{L_{c_0,\tau_A},\gamma B_{\tau_A}}=\bar q_c$.
The basis $\gamma B_{\tau_A}$ is again a symplectic basis of type $(1,d)$.
Hence, there are $\tau'\in\mathfrak h_2$ and an isomorphism
$\varphi\colon A_{\tau_A}\to A_{\tau'}$ of polarised abelian surfaces taking
$\gamma B_{\tau_A}$ to $B_{\tau'}$ \cite[Ch.~8]{BirkenhakeLange1992complex}.
Since $\varphi$ is an isomorphism of polarised abelian surfaces commuting with
$(-1)$, the bundle $\varphi_*M$ is symmetric whenever $M$ is, and
$\varphi(D)\in|\varphi_*M|$ for every $D\in|M|$, with
$\mult_{\varphi(x)}\varphi(D)=\mult_xD$. Hence
$q_{\varphi_*M}(\varphi(x))=q_M(x)$ for $x\in A_{\tau_A}[2]$. As $\varphi$
maps the point with $\gamma B_{\tau_A}$-coordinates $(u,v)$ to the point with
$B_{\tau'}$-coordinates $(u,v)$, we get
\[
\bar q_{\varphi_*L_{c_0,\tau_A},B_{\tau'}}=\bar q_{L_{c_0,\tau_A},\gamma B_{\tau_A}}=\bar q_c ,
\]
and Lemma~\ref{lem:qformula} gives $\varphi_*L_{c_0,\tau_A}=L_{c,\tau'}$. The
curve $\varphi(D_A)\in|L_{c,\tau'}|$ is symmetric and even with an ordinary
node at $0$, and $A_{\tau'}\simeq A_{\tau_A}$ satisfies (G1). By
Lemma~\ref{lem:parity} and Step~1, $\varphi(D_A)$ is the zero locus of the
section of $H^0_+(L_{c,\tau'})$ vanishing at $0$. Hence (G2) holds for $c$ on
$A_{\tau'}$.
 
To conclude the proof of the lemma, let $(A,H)$ be very general and $c\in\cC_d$. By (G1)
and Step~1, $s\neq0$ and the sections of $H^0_+(L)$ vanishing at $0$ form the
line spanned by $s$. By (G2) the quadratic term of $s$ at $0$ is
non-degenerate, so Step~2 applies and proves all claims except uniqueness.
For uniqueness, a section of $H^0_+(L)$ whose zero locus is singular at $0$
vanishes at $0$, so it is a multiple of $s$ by Step~1.
\end{proof}

We are ready to count all genus $d$ hyperelliptic curves on $(1,d)$ polarised surfaces as the following theorem shows.
\begin{thm}
\label{thm:GenusgCurves}
    Let $(A,L)$ be a very general $(1,d)$ polarised abelian surface. Then 
    \begin{enumerate}
        \item For $d=2$ there are precisely 12 hyperelliptic genus 2 curves in $|L|$.  
        \item For $d=3$ there are precisely 90 hyperelliptic genus 3 curves in $|L|$.  
        \item For $d=4$ there are precisely 192 hyperelliptic genus 4 curves in $|L|$.  
        \item For $d=5$ there are precisely 150 hyperelliptic genus 5 curves in $|L|$. 
        \item For $d=6$ there are precisely 36 hyperelliptic genus 6 curves in $|L|$.
        \item For $d\geq 7$ there are no hyperelliptic genus $d$ curves in $|L|$.
    \end{enumerate}
\end{thm}

\begin{proof}
    Let $f:C \to A$ be a local embedding of a hyperelliptic curve of genus $d$ in $A$ such that $f(C)$ is symmetric and $f(C)$ has a node at $0$. Then $f(C)$ passes through $2d$ points of $A[2]$ with multiplicity 1 by Lemma \ref{lem:odd_mult_points}, so $\mathcal{O}(f(C))$ is a symmetric line bundle with an even divisor $f(C)$ in its linear system such that $\#A_2^-(f(C)) = 2d$. Therefore $\cO_A(f(C))=t^*_{\overline c}L_0$ with $c\in\cC_d$, and by Lemma \ref{lem:parity} we have $f(C) = Z(s)$ for some $s \in H^0_+(t_{\overline{c}}^*L_0)$ vanishing at $0$. By Lemma \ref{lem:existence}, $s$ is unique up to a constant, and conversely every $c\in\cC_d$ gives such a curve. Hence the symmetric hyperelliptic curves of geometric genus $d$ with a node at $0$ in the algebraic class of $L$ correspond bijectively to elements of $\cC_d$. By Lemma \ref{lem:unique_curve}, every hyperelliptic curve of geometric genus $d$ in $|L|$ is a translate of exactly one of them. Each of them has exactly $d^2$ translates in $|L|$, and these are distinct because the curve has a unique node and so is not invariant under any non-zero translation. Therefore $|L|$ contains exactly $d^2\cdot\#\cC_d$ hyperelliptic curves of geometric genus $d$, which gives $12$, $90$, $192$, $150$ and $36$ for $d=2,\dots,6$.
    For $d\geq 7$ the result follows from Corollary \ref{cor:nocurvesofhighgenus}.
\end{proof}

\begin{rem}
    The numbers in Theorem \ref{thm:GenusgCurves} are precisely the ones found in \cite[p.10, Table 1]{BryanOberdieckPandharipandeYin2015Curve}. This shows in particular that the transversality condition holds in these cases, see \cite[Lemma 12]{BryanOberdieckPandharipandeYin2015Curve}. To compute the number of hyperelliptic curves up to translation on a very general abelian surface, one needs to divide the numbers in Theorem \ref{thm:GenusgCurves} by $d^2$. It is because curves have precisely one node, so they are not invariant under any translation from $K(L)$. In particular, for $d=g=6$ we give an alternative proof of the fact from \cite{GrossiMoretti2025explicit} that there is a unique genus 6 hyperelliptic curve on a general $(1,6)$ polarised abelian surface.
\end{rem}

\section{$(1,d)$ polarisation with an odd $d$}\label{SectOdd}
When $d$ is odd then for every symmetric line bundle $L$ there is a unique $c\in\frac12\Lambda$ such that $L$ is of characteristic $c$. In particular, for a symmetric line bundle all of its 16 translates by $2$-torsion points give rise to all 16 symmetric line bundles.   

Let $A$ be a $(1,d)$ polarised surface with an odd $d$ satisfying $5\leq d\leq 15$ and let $C$ be a hyperelliptic curve of genus $g=\frac{d+5}2$ with a symmetric local embedding $f:C\to A$ such that $\cO(f(C))$ is of characteristic 0 (for a fixed decomposition). Note that by \cite{KLM} such curves exist on abelian surfaces. When we write $d=2k+1$, then $g=k+3$.

We start with the following lemma. 

\begin{lem}
\label{lem:odd_nodes}
    Let $(A, L)$ be a $(1,d)$ polarised abelian surface with $d = 2k+1$, $2 \leq k \leq 7$, and let $D$ be a symmetric reduced odd divisor of geometric genus $g = k+3$ with $L=\cO(D)$. Assume that $L$ is of characteristic $0$. Set $$Z = \{x \in A_2^+(D) \mid \operatorname{mult}_xD > 0\}.$$ Then $\#Z \leq k-1$. Moreover, if $\# Z = k-1$, then the points of $Z$ are precisely the singularities of $D$, all of them are nodes, and the normalisation of $D$ is a hyperelliptic curve of genus $k+3$. 
\end{lem}

\begin{proof}

Let $\nu: C \to D$ be the normalisation of $D$. Then $$g(C) = p_a(D) - \Delta = d+1 - \sum_{x \in \operatorname{Sing}(D)}\delta_x.$$ We denote the lift of $(-1)_{A}$ to $C$ by $j$. Let $R$ be the ramification of $\pi: C \to C/\langle j \rangle$. Note that $\pi$ is $2:1$ so $R$ is equal to the number of fixed points of $j$. 
The Riemann-Hurwitz formula applied yields $$R = 2d - 2\Delta - 4g(C/\langle j \rangle)+4.$$
Let us estimate $R$. The map $\pi$ is branched only over the points of $\nu^{-1}(A[2] \cap D)$. By Proposition \ref{prop:twotorsionofdivisors} we know that $\# A_2^-(D) = 10$, so $$R \geq 10 + \sum_{x \in A_2^+(D)}b_x,$$ where $b_x$ is the number of $j$-invariant branches. 
Therefore, we get $$10 + \sum_{x \in A_2^+(D)}b_x \leq R = 2d - 2\Delta - 4g(C/\langle j \rangle)+4,$$ which is equivalent to $$\sum_{x \in A_2^+(D)}b_x+2\Delta +4g(C/\langle j \rangle) \leq 2d-6 =4k-4.$$ The last inequality can be written as $$\sum_{x \in Z}(b_x + 2\delta_x)+\sum_{y \in \operatorname{Sing}(D)\setminus Z}2\delta_y +4g(C/\langle j \rangle) \leq 4k-4.$$ 

We claim that $b_x+2\delta_x\geq4$
for every $x\in Z$. If $\delta_x\geq2$, this is immediate. Suppose therefore that $\delta_x=1$. Then $x$ is either a node or a cusp. The latter cannot occur since any branch through a $2$ torsion point that is invariant under $(-1)_A$ must have odd multiplicity, while a cusp has multiplicity two. Thus, $\delta_x=1$ implies that $x$ is a node, hence $b_x = 2$, since two branches at a node are preserved by $(-1)_A$, and the claim is established. Therefore, $\#Z \leq k-1$.

If $\#Z=k-1$, then we have $$k-1=p_a(D)-g(C)=\Delta\geq\sum_{x \in Z}\delta_x\geq \#Z=k-1$$ so  $\Sing(D)\setminus Z = \varnothing$ and for all $x\in Z$ we have $\delta_x=1$, hence the previous inequality forces $b_x=2$ and $g(C/\langle j \rangle) = 0$, so $j$ is the hyperelliptic involution, $D$ does not have singularities outside of $Z$ and all points in $Z$ are ordinary nodes.

\end{proof}

Before stating the main result let us state the first existence result.
\begin{lem}\label{lem:onecurve} For $1\leq k\leq7$   there exists a symmetric local embedding $f$ of a genus $k+3$ hyperelliptic curve $C$ into a general $(1,2k+1)$ polarised abelian surface such that $f(C)$ is an odd section of a line bundle of characteristic 0 with $k-1$ nodes in $2$-torsion points of $A$. 
\end{lem}
\begin{proof}
    By \cite{KLM} there exists a hyperelliptic curve of genus $k+3$ locally embedded into a $(1,2k+1)$ polarised abelian surface with nodes only. We can translate it in such a way that the image of a chosen Weierstrass point is $0$, hence the embedding is symmetric and since $2k+1$ is odd, we can translate by a 2-torsion point to get that $\cL=\cO(f(C))$ is of characteristic 0. Put $D=f(C)$ and write $D=\{\theta_C=0\}$ for some theta function $\theta_C$.

    Now, we have inequality $\#A_2^-(D)\geq10$ in Lemma \ref{lem:sing_count}, so by Proposition \ref{prop:twotorsionofdivisors} we have equality $\#A_2^-(D)=10$ and the curve is an odd section.
    \end{proof}

Using the quadratic form from the proof of Proposition \ref{prop:twotorsionofdivisors}, see also \cite[Prop 4.7.2]{BirkenhakeLange1992complex}, we call an ample symmetric line bundle $L(H,\chi)$ odd if $\#q_L^{-1}(-1)=10$. In our setting, the odd line bundles are of characteristics that belong to $S$. Recall that being odd is intrinsic and does not depend on the choice of the decomposition of a lattice. The set of all odd line bundles will be denoted by $\tilde{S}$. As usual if a symmetric line bundle is not odd, it is called even and in particular $L_0$ is even.

Let us define the following covering of the moduli space of $(1,d)$ polarised abelian surfaces with $d$ odd and $2\leq k\leq 7$.
$$X=\{(A,H,P): (A,H)\in\cA_{2}^{(1,d)}, P\subset \tilde{S},\  \#P=k-1\} $$
By \cite{BM} and, in particular, the paragraph following Lemma 3.4 there, the restriction of the action of the modular group $\Sp_4(\ZZ)$ is the action of \(\Sp_4(\ZZ_2)\simeq S_6\) on the \((k-1)\)-element subsets \(P\subset \tilde{S}\) and is transitive. Hence \(X\) is connected.
Note that the forgetful map $X\to \cA_2^{(1,d)}$ is finite of degree $\binom{6}{k-1}$ and $X$ is irreducible.

Let $D$ be a reduced irreducible curve being an odd section of a characteristic $0$ line bundle $L_0$ with $k-1$ nodes at 2-torsion points and geometric genus $k+3$, see Lemma \ref{lem:odd_nodes}. Let $z\in Z$ be a node. Then $t_{-z}(D)$ has a node at 0 and passes through ten 2-torsion points with multiplicity 1. In particular, $\cO(t_{-z}(D))$ is an odd line bundle. Therefore, we define $$\Psi(D)=\{\cO(t_{-z}(D)),z\in Z\}\subset\tilde{S}$$
Analogously, for $P\subset\tilde{S}$ we define $Z_P=\{z\in A[2]: t_z^*L_0\in P\}$. 
One can check explicitly that $$\Psi(D)=P\Leftrightarrow \Sing(D)=Z_P.$$

Now, we need a technical lemma that can be seen as a counterpart of Lemma \ref{lem:existence}.
\begin{lem}\label{moduliopennes}
    
    Let $U$ be the set of triples $(A,H,P)$ such that there exists an irreducible (geometric) genus $(k+3)$ curve $D$ being an odd section of a characteristic $0$ (for some decomposition) line bundle with $\Psi(D)=P$. Then $U$ is open. 
\end{lem}
\begin{proof}
Let $\theta_1,\ldots,\theta_k\in H^0_-(L_0)$ be the basis of the space of Riemann classical odd theta functions for the line bundle $L_0$ of characteristic 0. The set $A_2^+(D)$ does not depend on a choice of any $D$ that is a symmetric odd divisor with $\cO(D)=L_0$. Now, choose $Z=Z_P=\{z_1,\ldots,z_{k-1}\}\subset A_2^+(D)$. Then, there exists a non-zero section $\theta_Z=\sum_{i=1}^k \alpha_i\theta_i$ that satisfies the set of $k-1$ linear equations $\sum^{k}_{i=1} \alpha_i\theta_i(z_j)=0,\  j=1,\ldots, k-1$. 
Hence, we define $D=(\theta_Z=0)$. The solution is unique up to a scalar and $D$ is well defined because if two (symmetric) curves $D,D'$ have nodes at the same points then $D\cdot D'\geq 10+4(k-1)=4k+6 > 4k+2=L\cdot L$ that is a contradiction. 
Now, classical theta functions are holomorphic on the Siegel space, hence (locally) a solution to a family of linear equations is holomorphic too. Note that the system of linear equations is of maximal rank equal to $k-1$, hence open. The irreducibility of the curve and having nodes at 2-torsion points are also open dense conditions, hence $U$ is open. 
\end{proof}

We have the following theorem that computes the number of hyperelliptic curves on $(1,d)$ polarised surfaces.
\begin{thm}\label{thmodd}
Let $(A, L)$ be a general $(1,d)$ polarised abelian surface for $d=2k+1, 2\leq k\leq 7$. Then there are exactly $d^2\cdot\binom{6}{k-1}$ hyperelliptic curves of geometric genus $k+3$ in the linear system $|L|$. If $k>7$ then there are no hyperelliptic curves of geometric genus $k+3$ on $(1,2k+1)$ polarised abelian surfaces.
\end{thm}
\begin{proof}
    By a possible translation, we can assume that $L$ is symmetric of characteristic $0$. 
   We start by showing the upper bound.
   Since $d$ is odd and $K(L)\simeq\ZZ_d^2$ we see that only one of $d^2$ translations of a hyperelliptic curve maps Weierstrass points to $2$-torsion points of $A$. Denoting the curve by $D$, by Lemma \ref{lem:sing_count}, we get that $\#A_2^-(D)\geq 10$, so the curve has to be an odd section and $\#A_2^-(D)=10$ and we get an equality in Lemma \ref{lem:sing_count}. In particular, it has $k-1$ nodes in the remaining 6 points. This gives a contradiction for $k-1>6$, hence no such curves exist. Assume $k\leq 7$. If two curves $D,D'$ have nodes in the same points then $D\cdot D'\geq 10+4(k-1)=4k+6 > 4k+2=L\cdot L$, which is a contradiction. Therefore the number of hyperelliptic curves is at most  $d^2\cdot\binom{6}{k-1}$.

Now, let us construct these curves. Firstly, by Lemma \ref{lem:onecurve} the set $U$ from Lemma \ref{moduliopennes} is non-empty. Since the moduli space $X$ is irreducible and the complement of $U$ is an algebraic set, we get that $U$ is dense. The forgetful map is finite, so a standard argument shows that the set of abelian varieties for which the full fibre is contained in $U$ is open, hence dense. Now, choose such $A$ and consider the curves $D_P,\ P\subset \tilde{S},\ \#P=k-1$. By Lemma \ref{lem:odd_nodes}, the curves $D_P$ are hyperelliptic. 

Note that, according to the definition of $U$ it may happen that $D_P$ is a section of a line bundle of characteristic 0 for a different decomposition. Then, on the fixed decomposition, $D_P$ is a section of an $L$ of some characteristic $c$. Note that $L$ is even and $D_P$ is its odd section.
Then, $t_{-c}D_P$ is a section of $L_0$ (still odd because it passes through 10 points with odd multiplicity) and its set of nodes is $t_{-c}Z$. Now, by definition 
$$\Psi(t_{-c}D_P)=\{\cO(t_{-z}(t_{-c}D_P)), z\in t_{-c}Z\}=\{\cO(t_{-z}(D_P)), z\in Z\}=P$$   

This shows that for each set $P$ there exists a hyperelliptic curve $D$ being an odd section of $L_0$ with $\Psi(D)=P$, hence a general abelian variety contains precisely $d^2\cdot\binom{6}{k-1}$ hyperelliptic curves in the linear system of a line bundle of characteristic zero.
\end{proof}

\begin{rem}\label{BS}
    We discarded $k=1$, since then the geometric and arithmetic genera coincide, hence the curve is smooth. However, one can follow the arguments of this section to see it reproves the fact that there are 9 smooth hyperelliptic curves of genus 4 on a general $(1,3)$ polarised surface, see \cite{BorowkaSankaran2017hyperelliptic}. 
\end{rem}

\section{$(1,d)$ polarisation for an even $d$}\label{secteven}
In this section, we compute the number of hyperelliptic curves of geometric genus $k+3$ on a general $(1,2k)$ polarised abelian surface. Let us start with the following theorem.
\begin{thm}\label{nmbrdeven}
    Let $A$ be a very general $(1,d)$ polarised abelian surface with an even $d=2k\geq 4$. Then the number of hyperelliptic curves of geometric genus $k+3$ in the linear system of a line bundle $L$ equals $\frac{d^2}{4}\cdot\binom{4}{k-2}$ for $k\leq 6$ and there are no hyperelliptic curves of geometric genus $k+3$ for $k>6$. 
\end{thm}
\begin{proof}
    For $k=2$, we get smooth genus 5 hyperelliptic curves on a $(1,4)$ polarised abelian surface which is the content of \cite{BO20}. The authors show that there is a unique curve up to translation that is invariant under the Klein group $K(L)[2]$. By translating by points in $K(L)\setminus A[2]$ we get precisely 4 hyperelliptic curves in the linear system $|L|$.

    To show the remaining cases, let us assume that $f:C\to A$ is a symmetric local embedding of a hyperelliptic curve of geometric genus $k+3$ to a $(1,2k)$ polarised abelian surface with Weierstrass points mapped to $2$ torsion points. By Lemma \ref{lem:sing_count} we get that $\#A_2^-(f(C))\geq 12$. By Proposition \ref{prop:twotorsionofdivisors} the only possible choice is the equality $\#A_2^-(f(C))=12$, hence $f(C)$ can have nodes only at the remaining $2$-torsion points. In particular, the number of Weierstrass points of $C$ is at most $12+2\cdot 4=20$, so $g=k+3\leq9$ and $d\leq 12$.

Now, assume $3\leq k\leq 6$. By comparing geometric and arithmetic genera, we get that $f(C)$ has $k-2$ nodes. We further translate the curve (by a $2$-torsion point) so that it has a node at $0$. Then, by definition, it is an even divisor, hence by Proposition \ref{prop:twotorsionofdivisors} the line bundle is of characteristic $c=\frac{e_1}2+\frac{f_1}2$. Moreover, the proof of the proposition shows that $A_2^+(f(C))=K(L)[2]$. By intersection theory, if two (even, symmetric) curves $D,D'$ have nodes in the same subset of points of $K(L)$ they coincide because
$$D.D'\geq12+4(k-2)=4k+4=2d+4>2d=L^2$$
Note that translations by elements of $K(L)[2]$ permute the points of $K(L)[2]$. For $k=3,5$ the action is faithful and transitive, hence there is precisely one curve up to translation and translations by $K(L)$ give all the copies. This means that the number of curves in the linear system equals $d^2$ that agrees with the formula because $\binom{4}{1}=\binom{4}{3}=4$.

If $k=6$ the set of nodes is precisely $K(L)[2]$ and so the curve is unique and invariant under translations of $K(L)[2]$. Its translations by $K(L)\setminus K(L)[2]$ are also $K(L)[2]$ invariant, hence the number of curves in the linear system equals $\frac{d^2}4$.

For $k=4$ we can choose any $2$ points out of $4$ and translations by elements of $K(L)[2]$ either leave a curve invariant or swap the two points with their complement. In particular, the upper bound is again given by the formula $\frac{d^2}{4}\cdot\binom{4}{2}$.

The lower bound in cases $k\neq 4$ follows from the existence of a hyperelliptic curve given in \cite{KLM} because all hyperelliptic curves in the linear system are translations of this curve.
For $k=4$ we construct the curves in Theorem \ref{thm:1,8}. 
\end{proof}
The proof of Theorem \ref{nmbrdeven} for $k=6$ gives an immediate corollary.
\begin{cor}\label{corKlein}
    On a very general $(1,12)$ polarised abelian surface $(A,L)$ there is a unique (up to translation) hyperelliptic curve of geometric genus 9. The curve is an \'etale Klein covering of a genus 3 curve and the covering can be constructed as a restriction of the quotient isogeny $A\to A/(K(L)[2])$. In particular, the genus 3 curve is one of the curves in a linear system of a $(1,3)$ polarised surface, see Theorem \ref{thm:GenusgCurves}. 
\end{cor}

\begin{rem}
    The numbers presented in Theorem \ref{nmbrdeven} confirm the prediction of \cite[Table 1]{BryanOberdieckPandharipandeYin2015Curve}. By uniqueness, one can show that the curves are solutions to the following linear system of equations. If $\theta_1,\ldots,\theta_{k-1}$ is a basis for \(H^0_{+}(t^*_{\frac{1}{2}e_1 + \frac{1}{2}f_1}L_0)\), and $q_1,\ldots, q_{k-2}\in K(L)[2]$ is a set of points that are nodes then there exists a unique (up to a non-zero scalar) $k-1$ tuple \(\alpha_1,\ldots,\alpha_{k-1} \in \bC\) such that
		\begin{equation*}
			\alpha_1 \theta_1(q_j) + \alpha_2 \theta_2(q_j) +\ldots+ \alpha_{k-1} \theta_{k-1}(q_j) = 0\ \ \text{for all $j=1,\ldots, k-2$}
		\end{equation*}
        and $\theta=\alpha_1 \theta_1 + \alpha_2 \theta_2 +\ldots+ \alpha_{k-1} \theta_{k-1}$ is the theta function that gives the curve.  
\end{rem}

\subsection{The case $d=8$} In this subsection, we construct \(96\) hyperelliptic curves of geometric genus \(7\) in the linear system of the polarisation of a general \((1,8)\) polarised abelian surface. 
We show that these curves are double coverings of genus \(4\) curves constructed in \autoref{thm:GenusgCurves}.

The next proposition gives a crucial connection between hyperelliptic curves on $(1,4)$ and $(1,8)$ polarised surfaces. 
\begin{prop}
\label{lift14to18}
    Let $(B,L_B)$ be a very general $(1,4)$ polarised abelian surface with $B = V/\Lambda$ and $C$ be a symmetric hyperelliptic curve in $|L_B|$ of geometric genus 4 with a node at $0$. Then there exists a unique polarised isogeny $\pi: (A, L_A) \to (B, L_B)$ from a $(1,8)$ polarised abelian surface $(A,L_A)$ such that $\pi^{-1}(C)$ is a geometric genus 7 hyperelliptic curve with two nodes. More precisely, if $C = Z(s)$ for $s \in H^0_+(t_{\overline{c}}^*L_{B})$ for a characteristic $c$ with respect to some decomposition $\Lambda = \langle e_1, e_2\rangle \oplus \langle f_1, f_2 \rangle$ for $L_B$ then $\pi$ is given by $c+\frac12 e_1+\frac12 f_1 \in \frac{1}{2}\Lambda(L_B)$.

    The map $\pi^{-1}(C) \to C$ is an \'etale double cover defined by the line bundle $\mathcal{O}_{C}(w_1 - w_2)$, where $\{w_1, w_2\} = C \cap K(L_B)[2] \setminus \{0\}$. 
\end{prop}

\begin{proof}
    Fix a decomposition $\Lambda = \langle e_1, e_2\rangle \oplus \langle f_1, f_2 \rangle$ for $L_B$ with $\im H(e_1, f_1) = 1, \im H(e_2, f_2) = 4$. It follows from the properties of $C$ that there exists a characteristic $c \in \frac12\Lambda(L_B) \setminus \{0, \frac12 e_1, \frac12 f_1, \frac12e_1+\frac12f_1\}$ such that $C = Z(s)$ for some $s \in H^0_+(t_{\overline{c}}^*L_{B})$. Let $\pi_{c'}:(A, L_A) \to (B, L_B)$ be an isogeny of degree 2 defined by $c' \in \widehat{B}[2] \setminus K(L_{\widehat{B}})[2]$ (note that we must take a $2$ torsion point outside of $K(L_{\widehat{B}})[2]$ in order for $L_A$ to be of type (1,8)). Recall that the set $\widehat{B}[2] \setminus K(L_{\widehat{B}})[2]$ is in bijection with $\frac12\Lambda(L_B) \setminus \{0, \frac12 e_1, \frac12 f_1, \frac12e_1+\frac12f_1\}$ via the polarisation morphism $\varphi_{L_B}$, hence we regard $c'$ as an element of $\frac12\Lambda(L_B) \setminus \{0, \frac12 e_1, \frac12 f_1, \frac12e_1+\frac12f_1\}$. Then $\pi_{c'}^{-1}(C) = Z(s')$ for some $s' \in H^0(t^*_{\overline{c}}L_A) = H^0(t^*_{\overline{c - c'}}L_A)$. Note that any hyperelliptic curve on $A$ of geometric genus 7 with two nodes must pass through twelve $2$ torsion points of $A$ with multiplicity 1, hence in order for $\pi_{c'}^{-1}(C)$ to be as above we must have $c' = \frac{1}{2}e_1 + \frac{1}{2}f_1 + c$ (see Proposition \ref{prop:twotorsionofdivisors}). 
    
    It remains to prove that for this $c'$ the lift is indeed a hyperelliptic curve of geometric genus 7 with two nodes. Since $\pi_{c'}$ is \'etale we only need to show that $\pi_{c'}^{-1}(C)$ is hyperelliptic. It is irreducible because its class is a primitive polarisation. Recall that the curve $C$ has multiplicity 2 at 0, multiplicity 1 at 8 points of $B[2]$ and multiplicity 0 at all other points of $B[2]$, and it follows from Proposition \ref{prop:twotorsionofdivisors} that $C \cap K(L_B)[2] = \{w_1, w_2, 0\}$. Therefore, the intersection $C \cap (B[2]\setminus K(L_B)[2])$ consists of $6$ points, hence $\pi_{c'}^{-1}(C)$ has multiplicity 1 at all 12 points of $\pi_{c'}^{-1}(C \cap (B[2]\setminus K(L_B)[2])) \subset A[2]$ and nodes at $\{0, \kappa_1\} = \ker \pi_{c'}$. It follows from Riemann-Hurwitz formula applied to (the normalisation of) $\pi_{c'}^{-1}(C) \to \pi_{c'}^{-1}(C) / (-1)$ that $\pi_{c'}^{-1}(C) / (-1) \simeq \mathbb{P}^1$ and thus $\pi_{c'}^{-1}(C)$ is hyperelliptic. Moreover, we get that $w_1, w_2$ are the only Weierstrass points on $C$ whose preimages under $\pi_{c'}$ are not Weierstrass points of $\pi_{c'}^{-1}(C)$ so $\mathcal{O}_C(w_1 - w_2)$ is the defining line bundle of the cover $\pi_{c'}$, see \cite[Prop 4.2 and Figure 1]{BorowkaOrtega2019hyperelliptic}. 
\end{proof}

Now, we are able to compute the number of hyperelliptic genus 7 curves in the following theorem.
\begin{thm}
\label{thm:1,8}
	Let \((A,L)\) be a very general \((1,8)\)-polarised abelian surface. Then there are \(96\) geometric genus \(7\) hyperelliptic curves in \(|L|\) that are double covers of geometric genus 4 curves locally embedded into $(1,4)$ polarised surfaces.
\end{thm}

\begin{proof}
    Note that there are precisely three isogenies $(A,L) \to (B,L_B)$ such that $L_B$ is of type $(1,4)$, namely the quotients $\pi_i: A \to A/\left<p_i\right>$ for $i=1,2,3$, where $\{p_1, p_2, p_3, 0\} = K(L)[2]$. Moreover, by the assumptions on $A$, the surface $A/\left<p_i\right>$ is very general. It follows from the proof of Proposition \ref{lift14to18} that for $i=1,2,3$ the map $\pi_i$ gives a unique geometric genus 7 hyperelliptic curve $C_i$ with nodes at $0$ and $p_i$ in the linear system of $L$. The curves are invariant under translation by $p_i\in K(L)$ (and not by any other translation due to having two nodes at $0,p_i$), hence the total number of hyperelliptic curves equals $32\cdot 3=96$. Since, by Theorem \ref{nmbrdeven}, $96$ is also the upper bound, we get the result.
\end{proof}

\begin{rem}
Theorem \ref{nmbrdeven} and Theorem \ref{thm:1,8} show that all 96 hyperelliptic curves of geometric genus 7 in the linear system of a very general $(1,8)$ polarised surface are double covers of genus 4 curves that are locally embedded into $(1,4)$ polarised surfaces and the covering can be constructed as a restriction of quotient isogenies.
\end{rem}

We would like to finish the paper with the following problem that is a generalisation of the results of this section. Fix $k\in \ZZ$. Can one construct a hyperelliptic genus $g\geq k$ curve $C$ with $\mathbb{Z}/2\mathbb{Z} \times \mathbb{Z}/2\mathbb{Z} \subseteq \operatorname{Aut}(C)/\langle \iota \rangle$, where $\iota$ denotes the hyperelliptic involution on $C$, that can be locally embedded into an abelian surface? In other words, do there exist hyperelliptic Klein coverings of arbitrarily high genus locally embedded into abelian surfaces?

\section{Statements}
 The authors state that there is no conflict of interest.
 The manuscript has no associated data.

\subsection*{Statement on AI use}
All mathematical results, arguments, and proofs in this article are the work of the authors. The authors acknowledge the use of LLMs to check the final version of the manuscript for possible mathematical errors and typographical mistakes. Following LLM review, the exposition of the article has been improved.

\bibliographystyle{myalphaspecial} 
\bibliography{references} 

\end{document}